\documentclass[reqno,11pt]{article}
\usepackage{amssymb,amsthm,amsmath,amsfonts}
\usepackage{epsfig}

\theoremstyle{plain}
\begingroup

\newtheorem{thm}{Theorem}[section]
\newtheorem{theorem}{Theorem}[section]

\newtheorem{corollary}[thm]{Corollary}

\newtheorem{prop}[thm]{Proposition}
\newtheorem{proposition}[thm]{Proposition}

\endgroup
\theoremstyle{definition}
\newtheorem{defn}{Definition}[section]

\newtheorem{example}[defn]{Example}

\theoremstyle{remark}

\numberwithin{equation}{section}
\numberwithin{figure}{section}
 \DeclareMathOperator{\im}{Im}

\DeclareMathOperator*{\res}{\mathrm{Res}}

\def\D{{\mathbb D}}
\def\R{{\mathbb R}}
\def\C{{\mathbb C}}

\def\N{{\mathbb N}}

\def\ol{\overline}

\begin{document}

\title{Finite term relations for the exponential orthogonal polynomials}

\author{
Bj\"orn Gustafsson\textsuperscript{1},
Mihai Putinar\textsuperscript{2}\\
}

\maketitle

\begin{abstract} The exponential orthogonal polynomials encode via the theory of hyponormal operators a shade function $g$ supported by a bounded planar shape. We prove under natural regularity assumptions that these complex polynomials satisfy a three term relation if and only if the underlying shape is an ellipse carrying uniform black on white. More generally, we show that a finite term relation among these orthogonal polynomials holds if and only if the first row in the associated Hessenberg matrix has finite support. This rigidity phenomenon is in sharp contrast with the theory of classical complex orthogonal polynomials. On function theory side, we offer an effective way based on the Cauchy transforms of $g, \ol{z}g, \ldots, \ol{z}^dg$, to decide whether a $(d+2)$-term relation among the exponential orthogonal polynomials exists; in that case we indicate how the shade function $g$ can be reconstructed from a resulting polynomial of degree $d$ and the Cauchy transform of $g$. A discussion of the relevance of the main concepts in Hele-Shaw dynamics completes the article.

\end{abstract}

\noindent {\it Keywords: Complex orthogonal polynomials, exponential transform, finite term relation, hyponormal operator, quadrature domain} 

\noindent {\it MSC Classification: 33C47, 32A26, 47B20, 47B35}

 \footnotetext[1]
{Department of Mathematics, KTH, 100 44, Stockholm, Sweden.\\
Email: \tt{gbjorn@kth.se}}
\footnotetext[2]
{Department of Mathematics, UCSB, Santa Barbara, CA 93106-3080, USA and School of Mathematics, Statistics and Physics,
Newcastle University, Newcastle upon Tyne, NE1 7RU, UK.\\
Email: \tt{mputinar@math.ucsb.edu, mihai.putinar@ncl.ac.uk}}

\section{Introduction} In contrast to the unanimously accepted conventions in the theory of orthogonal polynomials on the line or on the circle, the multivariate setting
is asking for additional, rather essential choices, such as working with complex variables rather than real ones, ordering the monomials, smoothness properties of the
generating measures, (re)normalizations in the asymptotic analysis. It is well known for instance that the finite central truncation approach (of Jacobi matrices) to orthogonal polynomials
encounters serious complications when passing to several variables, see for instance \cite{DXu}. These difficulties are mostly related to the non-commutativity of the finite rank 
compressions of commuting, infinite (Jacobi type) matrices.

A rather specialized setting, relevant for instance for image processing, is offered by the theory of hyponormal operators. More specifically, a spectral parameter known as the principal function
of a hyponormal operator, turns the moments of a``shade function" depending of two real variables into a meaningful and accessible Hilbert space operator. The associated complex orthogonal polynomials, called 
in \cite{GP} and henceforth the {\it exponential orthogonal polynomials}, reveal in algebraic terms the geometry of the generating shaded shape. For instance, the analogs of finite collections of intervals on the line
or on the circle are quadrature domains for complex analytic functions in the plane. The analogy to the one dimensional classical theory, and especially to Markov's exponential transform of the generating function
of a moment sequence, goes quite far. The recent lecture notes \cite{GP} contain a detailed account of the qualitative features of exponential orthogonal polynomials and the necessary Hilbert space or
potential theory background.

The present note is aimed at proving a notable rigidity of the Hessenberg matrix associated to a system of exponential orthogonal polynomials. Namely, assuming a necessary completeness of the exponential orthogonal polynomials, a finite number of non-zero entries on the first
row propagates to the whole matrix and produces a finite banded matrix with one non-zero sub-diagonal. Finite banded matrices have been intensively studied for a long time, both for theoretical features as
well as for their relevance in numerical linear algebra, see for instance \cite{BG}.

Moreover, if only the first two elements of the first row of Hessenberg's matrix are non-zero and the exponential orthogonal polynomials are complete, then
the whole matrix is tridiagonal and the associated shaded shape is a an ellipse (black on white). This observation complements existing 
similar results for classical complex orthogonal polynomials \cite{D,MA,PS,Sz}. 

In the finite term relation scenario, the exponential transform of the generating function of the power moments of a shade function $g$ satisfies a characteristic vanishing formula. More specifically, the
Taylor coefficients at infinity of this transform reveal elementary recursion relations (at most quadratic) which populate the entire double series from its marginals. As a byproduct, we indicate an algorithmic approach to determine from the Cauchy transforms of $g, \ol{z} g, \ldots, \ol{z}^d g$ the existence of a finite term relation for the associated exponential orthogonal polynomials, and if so, how to reconstruct $g$.

While the main body of the present work is presented from a stationary point of view, a natural question raised in the last section addresses the evolution of shade functions (possibly carrying a finite term relation) under Hele-Shaw flow. The better understood case of families of confocal ellipses offers a promising base to start such a quest.

\tableofcontents

%%%%%%%%%%%%%%%%%%%%%%%%%%%%%%

\section{Preliminaries}\label{sec:preliminaries} 

The starting point of the dictionary between shade functions in the plane and Hilbert space operators is the following double Cauchy transform positivity feature. 
Let $g :\C \longrightarrow [0,1]$ be a Borel measurable function with compact support, and denote by 
$$ 
a_{jk} =\frac{1}{\pi} \int_\C z^j \ol{z}^k g(z) dA, \ \ j,k \geq 0,
$$
its moments. Above $dA$ stands for Lebesgue measure in the plane.
The double Cauchy transform
$$ 
-\frac{1}{\pi} \int_\C \frac{ g(\zeta) dA(\zeta)}{(\zeta-z) (\ol{\zeta}-\ol{w})} 
$$
expands in the convergent generating series in a neighborhood of infinity:
$$ 
- \sum_{j,k=0}^\infty \frac{a_{jk}}{z^{j+1} \ol{w}^{k+1}}.
$$
The exponential of this series, called the ``exponential transform" of the shade function $g$ has the remarkable positivity property of 
factoring through a Hilbert space $H$ and a linear bounded operator $T$ acting on $H$:
$$ 
E_g(z,w) = \exp [-\frac{1}{\pi} \int_\C \frac{ g(\zeta) dA(\zeta)}{(\zeta-z) (\ol{\zeta}-\ol{w})} ] = 1- \langle (T^\ast - \ol{w})^{-1} \xi, (T^\ast - \ol{z})^{-1} \xi \rangle.
$$
The latter identity extends from a neighborhood of infinity to the whole $\C^2$, with proper definitions of the localized resolvents. In this factorization the operator
$T$ is almost normal, in the precise sense that its self-commutator is small, but not zero:
\begin{equation}\label{TstarT}
[T^\ast, T] = \xi \langle \cdot, \xi \rangle =: \xi \otimes \xi.
\end{equation}
And vice-versa, an irreducible operator with non-negative self-commutator of rank one produces by the above exponential transform a shade function $g$, called the principal function.
For all aspects of the theory of hyponormal operators invoked in this article we refer to \cite{GP} and the bibliographical references therein.

The {\it exponential orthogonal polynomials} are trimmed to the Hilbert space realization rather than the function theoretic setting. Specifically, $P_k(z)$ is the unique polynomial of degree $k$
with positive leading coefficient which satisfies the state space orthogonality 
$$ 
\langle P_k(T^\ast)\xi, T^{\ast j}\xi \rangle = 0, \ \ j < k,
$$
and the normalization 
$$ 
\| P_k(T^\ast)\xi \| = 1, \ \ k \geq 0.
$$
This Krylov subspace type scheme goes through to any integer $k$ if and only if
the filtration of cyclic subspaces
\begin{equation}\label{Hk}
H_k = {\rm span}\{ \xi, T^\ast \xi, \ldots, T^{\ast k} \xi \}, \ k \geq 0,
\end{equation}
is not stationary.

The stationary case is interesting, and was studied in full detail before, being represented by black and white shade functions $g = \chi_\Omega$, where $\Omega$ is a quadrature domain for analytic functions in the complex plane, \cite{GP}.

All these being said, one can avoid the rather abstract operator $T$ and define a multiplier on a functional model space involving only Cauchy transforms and multiplier operations, see \cite{GP}
or the recent note \cite{GP-resolvent}. 
%This point of view will be adopted in a forthcoming section below.

%%%%%%%%%%%%%%%%%%%%%%%%%%%%%%

\section{Completeness of exponential orthogonal polynomials}\label{sec:completeness} 

We use the notation introduced in Section~\ref{sec:preliminaries}: $T \in L(H)$ is a pure hyponormal operator with rank-one self-commutator, $P_k(z)$
are the corresponding exponential orthogonal polynomials. In addition we let $\mathcal H$ represent the Hessenberg matrix associated to the multiplier $M_z$ with respect to the orthonormal system $P_k(z)$. 
Note that $P_k$ is defined without ambiguity if and only if the space $H_{k-1}$ (see (\ref{Hk})) is strictly contained in $H_k$. 
If this is not the case, that is $H_{k-1} = H_k$, then the spectrum of the operator $T$ is the closure of a quadrature domain (or open set) $\Omega$ of order less than or equal to $k$, 
carrying a ``maximal" principal function $g_T = \chi_\Omega dA$; moreover in this case $P_k$ makes no sense. 
We will carry this dichotomy, namely either $H_{k-1} \neq H_k$ for  all $k \geq 1$, or there exists $d\geq 1$, 
minimal with the property  $H_{d-1} = H_{d}$. In either situation we will speak about the exponential orthogonal polynomials $P_k$,
defined only for all  degrees $k \geq 0$, respectively only for  $0\leq k \leq d-1$. Correspondingly, the Hessenberg matrix $\mathcal H$ will be infinite or finite. 
More specifically, ${\mathcal H} = (h_{jk})$, where
$$ 
h_{jk} = \langle zP_k, P_j \rangle,
$$
with $j,k \geq 0$, respectively $0 \leq j,k \leq d-1$.
Note that the matrix $\mathcal H$ has only a first non-zero subdiagonal:
$$  
{\mathcal H}  = \left( \begin{array}{cccccc}
                                              h_{00} & h_{01} & h_{02} & h_{03} & h_{04} & \ldots\\
                                              h_{10} & h_{11} & h_{12} & h_{13} & h_{14} & \ldots\\
                                              0 & h_{21} & h_{22} & h_{23} & h_{24} & \ldots\\
                                              0& 0 & h_{32} & h_{33} & h_{34} & \ldots\\
                                              \vdots & & \vdots & & \ddots & \ddots\\
                                              \end{array} \right).
$$

 Let $H_{\rm pol}$ denote the closed linear subspace of $H$ generated by the orthonormal system $e_k = P_k(T^\ast)\xi$, $k \geq 0$:
$$
H_{\rm pol}= \bigvee_{k \geq 0} P_k(T^{\ast})\xi= \bigvee_{k \geq 0} T^{\ast k}\xi.
$$
 We may call $H_{\rm pol}$ the closure of exponential polynomials in the
 underlying Hilbert space $H$.  This subspace might be smaller than $H$, even finite dimensional, yet it carries full information on the operator $T$, and hence on the  associated shade function $g_T$.
 Indeed, the machinery of the exponential transform carries bijectively the Gram matrix data
 $$ 
b_{\ell,k} = \langle T^{\ast k}\xi, T^{\ast \ell}\xi \rangle, \ \ k,\ell \geq 0,
$$
 to the $(z,\ol{z})$-moments of $g_T$. One step further, standard linear algebra (LU factorization of the inverse Gram matrix) produces from $(b_{\ell,k})$ the coefficients of the orthogonal polynomials
 $P_k$ (for all degrees they are well defined) and finally the ``shift" matrix $\mathcal H$ together with the length of the vector $\xi = \| \xi\| e_0$ determine the Gram matrix, and hence $T$. We do not expand here these known details, see for instance \cite{B}.
 
The sub-diagonal in the Hessenberg matrix is of particular interest, at least in the general theory of orthogonal polynomials. We confine ourselves to note the simple dependence of its entries and the leading coefficients of the orthogonal polynomials. Specifically, if
$$ P_n(z) = \gamma_n z^n + \ldots,$$
where $\gamma_n >0, \ n \geq 0,$ then
$$ h_{n+1,n} = \langle z P_n(z), P_{n+1}(z) \rangle = \langle  \gamma_n z^{n+1} + \ldots , P_{n+1}(z) \rangle = \frac{\gamma_n}{\gamma_{n+1}}.$$
Recall also Christoffel function type interpretation
$$ \frac{1}{\gamma_n} = {\rm dist}(z^n, H_{n-1}),$$
which is the starting point of asymptotic formulas in the classical situation.
 
 As a conclusion of this general discussion, we emphasize the following rather striking phenomenon specific to our framework. Let $\pi : H \longrightarrow H_{\rm pol}$ denote the orthogonal projection. The operator $T^\ast$ leaves invariant the subspace 
 $H_{\rm pol}$, hence $T^\ast \pi = \pi T^\ast \pi = {\mathcal H}$. In matrix form, we obtain the block decomposition:
$$ 
T =  \left( \begin{array}{cc}
                {\mathcal H}^\ast & 0 \\
                B & C \\
               \end{array} \right).
$$   
The upper-left corner ${\mathcal H}^\ast $ determines the whole operator $T$ in a rather constructive manner (see the stair-case block-matrix structure discussed in \cite{GP}
for the case of quadrature domains, that is $\dim H_{\rm pol} < \infty$). In general the recursive procedure to compute all moments 
$$ 
\langle T^p T^{\ast q} \xi, T^k T^{\ast \ell}\xi \rangle, \ \ p,q,\ell,k \geq 0,
$$
from the compressed data
$$
 b_{k,\ell} = \langle T^{\ast \ell} \xi, T^{\ast k}\xi \rangle, \ \ k, \ell \geq 0,
$$
is derivable from the basic commutation relation $T^\ast T = T T^\ast + \xi \otimes \xi$. See for details Section XII.3 in \cite{MP}.

From these observations we derive a simple quantitative criterion for the completeness of exponential orthogonal polynomials, 
that is for the equality $H_{\rm pol} = H$ to hold.

\begin{theorem} 
Let $T \in L(H)$ be an irreducible hyponormal operator with rank-one self-commutator and associated shade function $g$.
   Let $P_k(z)$ denote the associated exponential orthogonal polynomials
   and let ${\mathcal H} = (h_{jk})$ be the associated Hessenberg matrix (finite or not).
   
   The system of exponential orthogonal polynomials is complete if and only if one of the following conditions is satisfied:
   \begin{equation}\label{L2}
 \sum_{k\geq 1} |h_{0k}|^2 - |h_{10}|^2 = \frac{\| g \|_1}{\pi},
\end{equation}
or equivalently
\begin{equation}\label{Txi}
T\xi \in H_{\rm pol}.
\end{equation}
\end{theorem}                           

\begin{proof} We first claim that the closure of polynomials $H_{\rm pol} = \bigvee_{k \geq 0} T^{\ast k}\xi$ is equal to the full space $H$ if and only if
(\ref{Txi}) holds.
%$$ 
%T \xi \in H_{\rm pol}.
%$$
Indeed, assuming $T\xi \in H_{\rm pol}$, we infer by induction from the commutation relation $T T^\ast  =  T^\ast T- \xi \otimes \xi$ that
$T T^{\ast k} \xi \in H_{\rm pol}$ for all $k \geq 0$ (see equation (\ref{commutator}) below for some details). 
But his means that $H_{\rm pol}$ is a reducing subspace for the irreducible operator $T$, hence $H_{\rm pol} = H$.

Next, $T \xi \in H_{\rm pol}$ if and only if the orthogonal projection $\pi$ onto $H_{\rm pol}$ satisfies
$$ 
\pi T \xi = T \xi,
$$
or equivalently
$$ 
\| T \xi \| = \| \pi T \xi \| = \| {\mathcal H}^\ast \xi \|.
$$
Let $e_0 = \xi/\| \xi \|$ be the first orthonormal vector in the system associated to the exponential orthogonal polynomials.
The last equality becomes:
$$
\| T e_0 \|^2 = \| {\mathcal H}^\ast e_0 \|^2 = \sum_{k \geq 0} |h_{0k}|^2. 
$$
On the other hand the commutation relation $T^\ast T = T T^\ast + \xi \otimes \xi$ implies 
$$ 
\| T e_0 \|^2 = \| T^\ast e_0 \|^2 + |\langle e_0, \xi \rangle|^2 = |h_{00}|^2 + |h_{10}|^2 + \| \xi \|^2.
$$
Finally, the very definition of the exponential transform and its factorization through the resolvent of the operator $T^\ast$ yield:
$$
 \| \xi \|^2 = \frac{1}{\pi}\int_{\C} g dA.
$$
By putting together these computations we obtain the equivalence in the statement.
\end{proof}

\begin{corollary} 
The first row and first column in the Hessenberg matrix attached to any irreducible hyponormal operator with rank-one self-commutator and associated shade function $g$
satisfy the inequality:
$$  
\sum_{k\geq 1} |h_{0k}|^2 - |h_{10}|^2 \leq \frac{\| g \|_1}{\pi}.
$$
\end{corollary}

For the proof it suffices to note that 
$$
 \| \pi T \xi \| \leq \| T \xi \|.
$$

Besides the analytic criterion offered by the theorem above, we can derive in purely geometric terms sufficient conditions for the completeness of exponential orthogonal polynomials.
\begin{proposition}\label{prop:blackwhite} 
Let $g = \chi_\sigma$ be a black and white shade function of compact support $\sigma$. The associated exponential orthogonal polynomials are complete
if the complement of $\sigma$ is connected and either ${\rm int\,} \sigma = \emptyset$ or $\partial \sigma$ is real analytic smooth and ${\rm int\,} \sigma$ is not a quadrature domain.
\end{proposition}

\begin{proof} We exploit the weak continuity of the generalized resolvent $(T^\ast - \ol{z})^{-1}\xi$, defined over the whole complex plane, \cite{GP}. First we eliminate the case
that ${\rm int\,} \sigma$ is a quadrature domain, characterized by the finite dimensionality of the space $H_{\rm pol}$.

Assume $H_{\rm pol} \neq H$, that is there exists a non-zero vector $u \in H \ominus H_{\rm pol}$. Consider the continuous function
$$ 
\alpha(z) = \langle u, (T^\ast - \ol{z})^{-1}\xi \rangle, \ \ z \in \C.
$$
The analyticity of $f(z)$ in the complement of $\sigma$ and a Neumann series at infinity imply
\begin{equation}\label{fzero} 
\alpha(z) = 0, \ \ z \in \C \setminus \sigma.
\end{equation}

In case ${\rm int\,}\sigma = \emptyset$, we infer from continuity that $f$ is identically zero. On the other hand, the span of the vectors $(T^\ast - \ol{z})^{-1}\xi, \ z \in \C,$
is dense in $H$ (see again \cite{GP}), hence $u = 0$, a contradiction.

In case $\Omega := {\rm int\,}\sigma $ is not empty and $\Gamma = \partial \sigma$ is real analytic smooth we argue as follows. 
The function $\alpha(z)$ defined above vanishes on the complement of $\sigma$ and it is
bianalytic in ${\rm int\,}\sigma$. In general $\alpha$ is bianalytic (i.e., $\partial^2 \alpha/\partial \bar{z}^2=0$) in regions where $g^2=g$, as can be seen by
examining the exponential transform. That is
$$ 
\alpha(z) = \ol{z} h_1(z) + h_2(z), \ \ z \in \Omega,
$$
with $h_1,h_2$ analytic functions in $\Omega$. 
In view of (\ref{fzero}):
\begin{equation}\label{f0}
\alpha= 0 \ \  \text{on}\ \ \Gamma.
\end{equation}

If $h_1$ is identically zero, then so is $h_2$ by the maximum principle and there is nothing left to prove. From now on we assume that $h_1$ is not identically zero.

The analyticity of $\Gamma$ also implies the estimate
$$ 
\int_\Omega \| \frac{\partial}{\partial\ol{z}} { (T^\ast - \ol{z})^{-1}\xi} \| \, dA(z) < \infty.
$$
For a proof, see for instance Theorem 5.6 in  \cite{GP-2003}. 
As a consequence of this bound one finds
$$ 
h_1 = \frac{\partial \alpha}{\partial \ol{z}} \in L^1(\Omega, dA),
$$
hence the function $h_2$ is also summable in $\Omega$.

A second implication of the analyticity of the boundary $\Gamma$ is the existence of a compact subset $K \subset \Omega$ and a positive constant $C$, so that
$$
 | \int_\Omega f dA | \leq C \| f \|_{\infty, K}, \ \ f \in {\mathcal O}(\ol{\Omega}),
$$
that is for every analytic function $f$ defined in a neighborhood of $\sigma = \ol{\Omega}$.
Indeed, due to the smooth analyticity of $\Gamma$ there exists an analytic function $S(z)$ defined in a neighborhood of $\Gamma$, subject to the constraint
$$ 
S(\zeta) = \ol{\zeta}, \ \ \zeta \in \Gamma.
$$
This is the so called ``Schwarz function" of the curve $\Gamma$ \cite{AS, Shapiro-1992}. Under these assumptions Stokes' theorem yields
$$ 
\int_\Omega f dA = \frac{1}{2i} \int_\Omega f(z) d\ol{z} \wedge dz = \frac{1}{2i} \int_\Gamma f(\zeta) S(\zeta) d\zeta,
$$
and the latter line integral can be pushed inside $\Omega$ by Cauchy's theorem and the analyticity of $S$.

Let $(\Omega_n)$ be an increasing exhaustion of $\Omega$ with relatively compact domains with smooth boundaries. For a function $f \in {\mathcal O}(\ol{\Omega})$ we find
$$ 
\int_\Omega  h_1 f dA = \lim_n \int_{\Omega_n} h_1 f dA = \lim_n \frac{1}{2i} \int_{\partial {\Omega_n}} \ol{z} h_1 f dz = 
$$ 
$$
\lim_n \frac{1}{2i} \int_{\partial {\Omega_n}} (\ol{z}h_1  + h_2)f dz = \frac{1}{2i} \int_\Gamma \alpha f dz = 0.
$$

Since the analytic function $h_1$ admits only finitely many zeros on $K$, there exists a polynomial $P(z)$ vanishing exactly at these zeros, counting also multiplicities.
We deduce
$$ \int_\Omega P f dA = 0, \ \ \ f \in {\mathcal O}(\ol{\Omega}).$$ 
Lagrange interpolation formula proves then that $\Omega$ is a quadrature domain.
\end{proof}

We will see shortly that one cannot avoid in the statement of the proposition the assumption on $\C \setminus \sigma$ to be connected. A challenging question is to derive similar 
sufficient conditions for the equality $H_{\rm pol}=H$ in the presence of a true shade function $g$ (i.e. $g^2 \neq g$). One step further, one can refine the above criterion to
a non-simply connected domain, by replacing the polynomial closure $H_{\rm pol}$ by the larger space 
$$ H_{\rm an} = {\rm span}_{\lambda \notin \sigma(T)} (T^\ast - \ol{\lambda})^{-1}\xi.$$
This generalization steps outside the realm of orthogonal polynomials and we do not pursue it here.

%%%%%%%%%%%%%%%%%%%%%%%%%%%%%%%%%%%%%%%%%

\section{Three term relation}\label{sec:three term} 

In this section we prove among other things that ellipses are the only black and white shapes in the complex plane (with connected component and without slits) which carry a three term relation in the corresponding quantization by exponential orthogonal polynomials.

Let $(P_n(z))_{n=0}^\infty$ be the exponential orthogonal polynomials associated to a compactly supported shade function
$g : \C \longrightarrow [0,1]$ which is not the characteristic function of a quadrature domain. The Hessenberg matrix ${\mathcal H} = (h_{jk})$, is therefore infinite:
$$ 
h_{jk} = \langle zP_k, P_j \rangle, \ \ j,k \geq 0,
$$
and has only the first sub-diagonal non-zero. If only the first upper-diagonal is non-zero one encounters a classical {\it three term relation}
$$ 
z P_k(z) = h_{k+1,k} P_{k+1}(z) + h_{kk} P_k(z) + h_{k-1,k}P_{k-1}(z), \ \ k \geq 0,
$$
where $P_{-1}=0$ by convention. It is well known that orthogonal polynomial in the Lebesgue space $L^2$ of a positive measure supported by the real line satisfy
a three term relation. A necessary cosmetic modification also produces a three term relation for orthogonal polynomials on the circle.

Our main result reduces the three term relation to a much weaker recursion identity, identified by the first row of Hessenberg's matrix. Namely
$$ 
z P_k(z) = \sum_{j=1}^{k+1} h_{jk} P_j(z), \ \ k \geq 2,
$$
instead of having the summation starting with $j=0$. 
In other terms
$$ 
\langle z P_k(z), P_0(z) \rangle = 0, \ k \geq 2,
$$
or equivalently, Hessenberg matrix $\mathcal H$ has only the first two terms possibly non-zero on the first row:
$$ 
{\mathcal H}  = \left( \begin{array}{cccccc}
                                              h_{00} & h_{01} & 0 & 0 & 0& \ldots\\
                                              h_{10} & h_{11} & h_{12} & h_{13} & h_{14} & \ldots\\
                                              0 & h_{21} & h_{22} & h_{23} & h_{24} & \ldots\\
                                              0& 0 & h_{32} & h_{33} & h_{34} & \ldots\\
                                              \vdots & & \vdots & & \ddots & \ddots\\
                                              \end{array} \right).
$$

Recall that the first row of Hessenberg's matrix can be computed in the spirit of classical orthogonal polynomials by the formula:
$$ 
\langle z P_k(z), P_0(z) \rangle = \gamma \int_\C z P_k(z) g(z) \ {dA}(z),
$$ 
where $\gamma$ is a constant, \cite{GP} ($\gamma=1/(\pi \|\xi \|)$, necessarily).

The precise statement is:

\begin{theorem}\label{ellipse}  
Assume the exponential orthogonal polynomials $(P_n(z))_{n=0}^\infty$ associated to a compactly supported shade function
$g : \C \longrightarrow [0,1]$ are complete in the Hilbert space realization of the exponential transform of $g$, i.e. $H_{\rm pol}=H$.
If, in addition, the recurrence relation among $P_k(z)'s$ is missing the $0$-th order term for all $k \geq 2$, or equivalently if
\begin{equation}\label{zPg}
\int_\C z P_k(z) g(z) \, { d A}(z) = 0, \ \ k \geq 2,
\end{equation}
then $g$ is the characteristic function of an ellipse.
\end{theorem}

\begin{proof} Denote by $T \in L({H})$ the irreducible hyponormal operator $T$ with rank-one self-commutator
$ [T^\ast, T] = \xi \otimes \xi$ and with principal function equal, a.e., to the prescribed shade function $g$.
The assumption that the exponential orthogonal polynomials are complete is equivalent (by definition) to the fact that the filtration
of finite dimensional spaces
$$ 
H_k = {\rm span} \{ \xi, T^\ast \xi, \ldots, T^{\ast k} \xi\}, \ \ k \geq 0,
$$
does not stop and it spans  $H$. In other terms the system of orthonormal vectors $e_k = P_k(T^\ast)\xi$, $k \geq 0$, spans $H$.

Hessenberg's matrix $(h_{jk})_{j,k=0}^\infty$ represents the linear transform $T^\ast$ in the basis $(e_k)$, while its conjugate transpose represents
$T$. The assumption (\ref{zPg}) and the completeness of orthogonal polynomials translate into 
$$ 
T \xi \in H_1,
$$
or equivalently
\begin{equation}\label{TabT}
T\xi = a \xi + b T^\ast \xi,
\end{equation}
for some complex numbers $a,b$. Induction based on the identity
\begin{equation}\label{commutator}
T T^{\ast k} = T^{\ast k}T + \sum_{j=0}^{k-1} T^{\ast j} [T,T^\ast] T^{\ast k-j-1}=
T^{\ast k}T - \sum_{j=0}^{k-1} T^{\ast j} \xi \langle \cdot, T^{k-j-1}\xi \rangle
 \end{equation}
implies, together with (\ref{TabT}), that
$$ 
T H_k \subset H_{k+1}, \ \ k \geq 0.
$$
In other terms, Hessenberg's matrix is tri-diagonal. According to Theorem 5.2 of \cite{GP}, the shade function $g$ coincides then with the characteristic function of an ellipse.

We repeat for completeness the simple argument contained in \cite{GP}. Denote 
$$
a_k = \overline{h_{kk}},\quad b_{k+1} = \overline{h_{k,k+1}}, \quad c_{k+1} = \overline{h_{k+1,k}}, \quad k \geq 0,
$$
the three non-trivial diagonals in Hessenberg's matrix. The commutator assumption $[T^\ast,T] = \xi \otimes \xi$ reads on matrix elements:
$$
 |a_0|^2 + |b_1|^2 = r + |a_0|^2 + |c_1|^2, 
$$
$$
 |c_1|^2 + |a_1|^2 + |b_2|^2 = |b_1|^2 + |a_1|^2 + |c_2|^2, \ldots,$$
$$ 
\overline{a_0}c_1 + \overline{b_1} a_1 = a_0 \overline{b_1} + c_1 \overline{a_1},
$$
$$
 \overline{a_1}c_2 + \overline{b_2} a_2 = a_1\overline{b_2} + c_2 \overline{a_2}, \ldots   
$$
and
$$ 
\overline{b_1} c_2 = c_1 \overline{b_2}, \quad \overline {b_2} c_3 = c_2 \overline{b_3}, \ldots .
$$
We have denoted  $r = \|\xi\|^2 >0.$

Elementary algebra implies 
$$ 
|b_1| = |b_2| = |b_3| = \ldots; \ \ |c_1| = |c_2| = |c_3| = \ldots.
$$
In virtue of the assumption that the chain of Krylov subspaces $(H_k)$ is not stationary, the entries $c_k$ are non-zero.

Then a unitary  transformation $T \mapsto U^\ast T U$ with $U$ diagonal reduces lower-diagonal entries to a positive constant:
$$ 
b_1 = b_2 = b_3 = \ldots = s>0.
$$
Again simple algebra yields
$$ 
c_1 = c_2 = c_3 = \ldots = u \in \C,
$$
and
$$ 
a_1 = a_2 = a_3 = \ldots = a.
$$
After a translation $T \mapsto T-aI$ and homotethy we can assume $a=0$ and $s = 1$ reaching after all the operator 
$$ 
T_1 = uS + S^\ast,
$$
where $S$ denotes the unilateral shift on $\ell^2(\N)$. But then it is known, and easy to verify, that
the principal function $g_{T_1}$ is the characteristic function of an ellipse. The affine transformation can be reversed
at both levels, operator, respectively, principal function, and thus completes the proof.

We refer to Section 5.3 in \cite{GP} for complete details.
\end{proof}

The reader will recognize in the above proof a purely matrix algebra phenomenon. To be more specific, let ${\mathcal H} = (h_{jk})_{j,k=0}^\infty$
denote a Hessenberg matrix with a single non-trivial subdiagonal consisiting of non-null elements. Assume that ${\mathcal H}$
induces a linear bounded operator on $\ell^2(\N)$ and denote by $\{ e_0, e_1, e_2, \ldots \}$ the canonical orthonormal basis of $\ell^2(\N)$. Theorem~\ref{ellipse}
states that the apparently innocent conditions
$$ 
[{\mathcal H}, {\mathcal H}^\ast] = \alpha \, e_0 \otimes e_0,
$$
and 
$$ 
{\mathcal H}^\ast e_0 = \beta e_0 + \gamma e_1,
$$
where $\alpha >0$, $\beta$, $\gamma$ are constants, imply ${\mathcal H}$ is a Toeplitz matrix with at most three non-trivial diagonals: the main diagonal, the first sub-diagonal and the first upper-diagonal.
In other terms: ${\mathcal H} = c_1 S + c_0 + c_{-1} S^\ast$, where $S$ denotes the unilateral shift and $c_{-1}, c_0, c_1$ are constants.

A simple adaptation of the commutator identity (\ref{commutator}) yields the following consequence.

\begin{corollary}\label{finite-term}
Let $(P_n(z))_{n=0}^\infty$ be a complete system of exponential orthogonal polynomials associated to a shade function
$g $. Let $d$ be a positive integer. If 
$$
\int_\C z P_k(z) g(z) \, {d A}(z) = 0, \ \ k \geq d+1, 
$$
then these orthogonal polynomials satisfy a $(d+2)$-recurrence relation:
\begin{equation}\label{FT}
z P_k(z) = \sum_{j=k-d-1}^{k+1} h_{jk} P_j(z), \ \ k \geq d+1.
\end{equation}
\end{corollary}

In other terms, if the first row in Hessenberg's matrix ${\mathcal H}$ has non-zero entries only on the positions $0,1,\ldots,d$, then 
the whole matrix will have only the first $d$ upper-diagonals non-zero. That is ${\mathcal H}$ is a finite banded matrix.

%%%%%%%%%%%%%%%%%%%%%%%%%%%%%%%%%%

\section{Algebra of exponential transform coefficients}\label{sec:algebra}

The finite term relation fulfilled by the exponential orthogonal polynomials implies, and can be derived, from purely algebra manipulations of the
Taylor coefficients at infinity of the underlying exponential transform. We expand in this section some of these computations bearing a numerical matrix analysis flavor.
Seen as a formal transform of the power moments of the generating shade function $g$, the Taylor series at infinity of the exponential transform encodes 
subtle geometric properties. For instance a specific rationality of this double series is equivalent to $g$ being the characteristic function of a quadrature domain.
Along these lines, Chapter 5 of \cite{GP} exploits a Pad\'e approximation scheme in two variables  of the exponential transform.

Let $T$ be an irreducible hyponormal operator with principal function $g$ and self-commutator $[T^\ast,T]= \xi \otimes \xi$. An expansion at infinity of the associated
exponential transform reads:
\begin{equation}\label{ab}
\exp ( - \sum_{j,k=0}^\infty \frac{a_{jk}}{z^{j+1} \ol{w}^{k+1}}) = 1 - \sum_{j,k=0}^\infty \frac{b_{jk}}{z^{j+1} \ol{w}^{k+1}},
\end{equation}
where $(a_{jk})$ are the power moments of $g$, while 
$$ 
b_{jk} = \langle T^{\ast k}\xi, T^{\ast j}\xi \rangle, \ \ j,k \geq 0.
$$
In this section we avoid the case of a quadrature domain, when not all polynomials $P_n$ are well defined.
Recall that
$$
 \langle P_j(T^\ast)\xi, P_k(T^\ast)\xi \rangle = \delta_{jk}.
$$

According to Corollary \ref{finite-term}, a finite term relation holds for the {\it complete} system of orthogonal polynomials $(P_j)$ if and only if
there exists a positive integer $d$ satisfying
$$ 
\langle T^\ast P_k(T^\ast)\xi , \xi \rangle = 0, \ \ k \geq d+1,
$$
or equivalently, the vector $T\xi$ is a linear combination of $\xi, T^\ast \xi, \ldots, T^{\ast d} \xi$.
Let $Q(z) = q_0 + q_1 z + \ldots q_d z^d$ be a polynomial of degree $d$ encoding this dependence:
\begin{equation}\label{OPrelation}
T\xi = Q(T^\ast) \xi.
\end{equation}
In other terms this equality holds if and only if the numerical certificates
\begin{equation}\label{certif}
 b_{m+1,0} = \langle T\xi, T^{\ast m}\xi \rangle = \langle Q(T^\ast)\xi, T^{\ast m}\xi \rangle = 
\sum_{k=0}^d q_k b_{m, k}\ \ , \ \ m \geq 0,
\end{equation}
are satisfied.

As a matter of fact, if identity (\ref{OPrelation}) is true, then the system of exponential orthogonal polynomials is complete.
Simply because the vectors $T^j T^{\ast k}\xi, \ \ j,k \geq 0,$ can be reduced via the commutation relation $[T^\ast,T] = \xi \otimes \xi$
to combinations of $\xi, T^\ast \xi, T^{\ast 2}\xi, \ldots$.

At the level of formal series, the transfer of these observations to the exponential transform
$$ 
E(z,w) = 1 - \sum_{j,k=0}^\infty \frac{b_{jk}}{z^{j+1} \ol{w}^{k+1}}
$$
is straightforward. Indeed, the coefficient of $\frac{-1}{z^{m+1}\ol{w}}$ in $zE(z,w)$ is
$b_{m+1,0}$, while the same coefficient in $Q(\ol{w})E(z,w)$ is
$\sum_{k=0}^d q_k b_{m, k}$.

In conclusion, we have proved the following criterion.

\begin{proposition}
Let $(P_n(z))_{n=0}^\infty$ be a complete system of exponential orthogonal polynomials associated to a shade function
$g$ and denote by $E(z,w)$ the respective exponential transform. Let $d$ be a positive integer. 

The system  $(P_n)$ satisfies a $(d+2)$-term relation (\ref{FT})
if and only if there exists a polynomial $Q(z)$ of degree $d$ with the property
\begin{equation}\label{res}
 \res_{w=0} [(z-Q(\ol{w}))E(z,w) ] = 0,
 \end{equation}
identically in $z$.
\end{proposition}

One step further, we prove in the conditions of the above Proposition that similar recurrence relations fill the full matrix of exponential transform coefficients $(b_{jk})_{j,k=0}^\infty$
having as input the marginal data $(b_{j,0})_{j=0}^\infty$. Indeed, we start from
$$ T^{\ast n} T \xi = T^{\ast n} Q(T^\ast) \xi,$$
where $n$ is a positive integer. The commutation relation imposed on $T$ and $T^\ast$ implies
$$ T^{\ast n} T = T T^{\ast n} + \sum_{j=0}^{n-1} T^{\ast j} [T^\ast,T] T^{\ast (n-1-j)} = T T^{\ast n} + \sum_{j=0}^{n-1} T^{\ast j}\xi \langle \cdot,  T^{(n-1-j)}\xi \rangle.$$
Let $m$ be a non-negative integer. Then
$$ \langle T T^{\ast n}\xi  , T^{\ast m}\xi \rangle = \langle T^{\ast n} Q(T^\ast) \xi, T^{\ast m}\xi \rangle - \sum_{j=0}^{n-1} \langle T^{\ast j}\xi, T^{\ast m}\xi \rangle \langle \xi,  T^{(n-1-j)}\xi \rangle.$$

Hence we can state the following complement to the finite term criterion.

\begin{corollary}\label{cor:recursion} 
Let $(P_n(z))_{n=0}^\infty$ be a complete system of exponential orthogonal polynomials associated to a shade function
$g$ and denote by 
$$ 
E(z,w) = 1 - \sum_{j,k=0}^\infty \frac{b_{jk}}{z^{j+1} \ol{w}^{k+1}}
$$ 
the respective exponential transform. Fix a positive integer $d$.

If there are complex constants $q_0, q_1, \ldots, q_d$ satisfying
\begin{equation}\label{first-column}
 b_{m+1,0} = \sum_{k=0}^d q_k b_{m, k}, \ \ m \geq 0,
 \end{equation}
then
\begin{equation}\label{all-columns}
 b_{m+1,n} = \sum_{k=0}^d q_k b_{m, k+n} - \sum_{j=0}^{n-1} b_{mj} b_{0, n-j-1},
 \end{equation}

for all $m \geq 0$ and $n \geq 1$.

\end{corollary}

 In function theory terms, we rephrase the above computations into the following statements invoking a shade function $g$ which carries a complete system of
 exponential orthogonal polynomials:
 \bigskip
 
{\bf 1).} Fix a positive integer $d$. The Cauchy transforms 
$$
F_k(z) = \frac{-1}{\pi} \int \frac{g(\zeta) \ol{\zeta}^k dA(\zeta)}{\zeta -z}, \ \ 0 \leq k \leq d,
$$
determine via condition (\ref{res}) whether the exponential orthogonal polynomials satisfy a $(d+2)$-finite term relation.

To be more precise, we seek a polynomial $Q(z)$ of degree $d$ with the property
\begin{equation}\label{Cauchy-transf}
  \res_{w=0} [(z-Q(\ol{w}))\exp( -\frac{F_0(z)}{\ol{w}} -\frac{F_1(z)}{\ol{w}^2} - \cdots - \frac{F_d(z)}{\ol{w}^{d+1}}) ] = 0, \ \ |z| >> 1.
  \end{equation}
\bigskip

{\bf 2).} In this case, Cauchy's transform $F_0(z)$ and the polynomial $Q(z)$ resulting from 1) determine the generating shade function $g$.
\bigskip

Notice that condition (\ref{Cauchy-transf}) does not imply, in general, that the system of exponential orthogonal polynomials is complete.
Rotationally invariant shapes, discussed in Section \ref{Rot} provide counterexamples.

The numerical analyst might feel betrayed by statement 2) above. A clarification is in order: our recursion formulas (\ref{all-columns}) determine 
all power moments of the function $g$. In particular one can estimate from there the location of the support of $g$, say in the unit square $ -1 \leq {\rm Re\,} z \leq 1$, 
$-1 \leq {\rm Im\,} z \leq 1.$ Then $g$ (a bounded measurable function) can be recovered by its Fourier expansion with respect to a basis of classical bi-variate 
orthogonal polynomials adapted to the square (such as Legendre polynomials), or via an inversion formula for the Laplace-Fourier transform of $g$.
We do not expand here these details.

%%%%%%%%%%%%%%%%%%%%%%%%%%%%%%%%%%%

\section{Examples}\label{sec:examples}

In this section we illustrate by means of a couple of examples a departure from the Bergman or Hardy space scenario \cite{PS,KS,D,Sz}. To be more specific, a finite term relation for the exponential orthogonal polynomials
does not imply in general that the underlying shape is an ellipse.

%%%%%%

\subsection{Two non-zero diagonals} 

We will construct a sparse infinite matrix $T$, bounded as a linear transformation on $\ell^2(\N)$ and satisfying
\begin{equation}\label{self-commutator}
 [T^\ast, T] = e_0 \otimes e_0,
 \end{equation}
 where $\{e_0, e_1, e_2, \ldots \}$ is the orthonormal basis with respect to which one expresses the matrix elements.

Specifically, we choose recursively real numbers $a_1, a_2, a_3, \ldots$ and $b_0, b_1, b_2, \ldots$ which populate the matrix
$$ 
T = \left( \begin{array}{ccccc}
                0& a_1&0&0& \ldots\\
                0&0&a_2&0& \ldots \\
                b_0&0&0& a_3 \\
                0&b_1& 0 & 0 & \ddots\\
                \vdots& & \vdots& & \ddots
                \end{array} \right).
$$
We display for convenience the adjoint:
$$ 
T^\ast =  \left( \begin{array}{ccccc}
                0& 0 &b_0&0& \ldots\\
                a_1&0&0&b_1& \ldots \\
                0&a_2&0& 0& b_2 \\
                0&0& a_3 & 0 & \ddots\\
                \vdots& & \vdots& & \ddots
                \end{array} \right).
$$
Condition (\ref{self-commutator}) is equivalent to 
   $$ b_0^2 - a_1^2 =1,$$
   $$ a_1^2 + b_1^2 = a_2^2,$$
   $$a_2^2 + b_2^2 = a_3^2+ b_0^2,$$
   $$a_n^2 + b_n^2 = a_{n+1}^2 + b_{n-2}^2, \ \ n \geq 2,$$
   and
   $$ b_0 a_3 = a_1 b_1,$$
   $$ b_1 a_4 = a_2 b_2,$$
   $$ b_{n-1}a_{n+2} = a_n b_n, \ \ n \geq 1.$$
   
   We denote for convenience $A_n = a_n^2, \ \ n \geq 1,$ and $B_n = b_n^2, \ \ n \geq 0.$
   Therefore
   $$ B_0 = A_1+1, \ B_1 = A_2-A_1,$$
   and in general
   $$ B_n = A_{n+1} - A_n + B_{n-2}, \ \ n \geq 2.$$
   Similarly,
   $$ B_{n-1}A_{n+2} = A_n B_n, \ \ n \geq 1.$$
   Note that addition yields
   \begin{equation}\label{sum-B}
    B_{n} + B_{n-1} = A_{n+1} + 1, \ \ n \geq 1.
    \end{equation}
   
   We first prove that by selecting the initial data $A_1>0$, $B_1>0$ the double sequence recurrence 
   goes through to infinity. Remark that $A_3 >0$ in this case. 
   
   Indeed, since division is involved we have to check that all $A_k's$ are non-zero.
   Assume $A_{k+3}=0$ and this is the first occurrence (for some $k>0$). We infer from
   $$A_{k+3}B_k = A_{k+1}B_{k+1} = 0$$
   that $B_{k+1} = 0.$ But then
   $$ A_{k+4}B_{k+1} = A_{k+2} B_{k+2} = 0$$ 
   implies $B_{k+2} = 0$. In conclusion we obtain a {\it finite} matrix
   $$ S = \left( \begin{array}{cccccc}
                0& a_1&0&0& \ldots &0\\
                0&0&a_2&0& \ldots  &0 \\
                b_0&0&0& a_3 & \ldots & 0\\
            \vdots& & \vdots& & \ddots & \vdots\\
            0 & \ldots & b_{n-1} & 0 & 0 & a_{n+2}\\
            0 & \ldots & 0 & b_n & 0 & 0 \\
                \end{array} \right)$$
   which satisfies $[S^\ast,S] = e_0 \otimes e_0$. But this is a contradiction to the fact that a commutator of
   finite matrices has vanishing trace. Or simply remark that identity (\ref{sum-B}) is violated.
   
   Thus the recurrence goes through to infinity and produces the two sequences of real numbers $(a_n)$, $(b_n)$
   and respectively the sequences of positive numbers $(A_n)$, $(B_n)$.

   Returning to the generating identities,  
  we can eliminate $B_n$ by multiplying (\ref{sum-B}) by $A_{n+2}$:
   $$ B_n A_{n+2} + B_n A_n = A_{n+2}+ A_{n+2}A_{n+1}$$
   whence
   $$ B_n = \frac{A_{n+2}+ A_{n+2}A_{n+1}}{A_n + A_{n+2}}, \ \ n \geq 1.$$
   On the other hand,
   $$ B_n = \frac{A_{n+1}}{A_{n+3}} B_{n+1}$$ yields
   $$ A_{n+2}(1+A_{n+1}) A_{n+3} (A_{n+1}+A_{n+3}) = A_{n+1} A_{n+3} (1+A_{n+2})(A_{n+1}+ A_{n+3}),$$
   or
   $$ \frac{A_{n+3}+A_{n+1}}{A_{n+2} + A_n} = \frac{1 + 1/A_{n+2}}{1+1/A_{n+1}}, \ \ n \geq 1.$$
   The telescopic product implies
   $$ A_{n+3} + A_{n+1} = \frac{A_3+A_1}{1+ 1/A_2} (1+ 1/A_{n+2}), \ \ n \geq 0.$$
   Denote $C =\frac{A_3+A_1}{1+ 1/A_2} $.
   
  It remains to prove that the sequence $(A_n)$ is uniformly bounded from above. In virtue of (\ref{sum-B}) the sequence $(B_n)$ would also be bounded
  and therefore $T$ will be a linear continuous transformation on $\ell^2$.
  
  Assume that there exists a subsequence $A_{n(j)}$ satisfying
  $$ \lim_j A_{n(j)} = \infty.$$ Identity
  $$ C(1+ \frac{1}{A_{n(j)+1}}) = A_{n(j)} + A_{n(j)+2}$$
  yields
  $$ \lim_j A_{n(j)+1} = 0.$$ 
Thus $B_{n(j)-1} + B_{n(j)} = A_{n(j)+1} + 1$ remains bounded as a function of $j$.
From the similar identity $B_{n(j)-2} + B_{n(j)-1} = A_{n(j)} + 1$ we find 
$$ \lim_j B_{n(j)-2} = \infty.$$
And $B_{n(j)-3} + B_{n(j)-2} = A_{n(j)-1} + 1$ implies
$$ \lim_j A_{n(j)-1} = \infty.$$
By repeating the above reasoning we find 
$$ \lim_j A_{n(j)} = 0,$$
a contradiction.
   
   Summing up, the following statement was just proved.
   
   \begin{prop} There exists a Borel measurable shade function $g : \C \longrightarrow [0,1]$ with compact support, so that the associated exponential
   orthogonal polynomials $P_n(z)$ satisfy the recursion
   $$ z P_n(z) = a_{n+1} P_{n+1}(z) + b_{n-2}P_{n-2}(z), \ \ n \geq 0,$$
   where $P_{-1} = P_{-2} = 0$ .
   \end{prop}
   
   In this case $g$ is not the characteristic function of an ellipse, nor a quadrature domain.

%%%%%%%%
   
   \subsection{Toeplitz operators} 

If we do not insist on only two non-zero diagonals, one can produce a similar counterexample as follows.
   Denote as before by $S$ the unilateral shift acting on $\ell^2(\N)$. Consider the linear operator
   $T = S^2 + S^\ast$. One verifies directly, for instance evaluating on the orthonormal basis $\{ e_0, e_1, e_2, \ldots\},$
   that $[S^\ast,S] = e_0 \otimes e_0$ and
   $$ [S^{\ast 2}, S^2] = e_0 \otimes e_0 + e_1 \otimes e_1.$$
   Hence
$$ 
[ T^\ast, T ] = [S^{\ast 2}, S^2] + [S,S^\ast] = e_1 \otimes e_1.
$$
   Thus $T$ is a hyponormal operator with rank-one self-commutator. And it satisfies our finite term relation constraint.
   Indeed,
   $$ T e_1 = e_3 + e_0,$$
   and
   $$ T^\ast e_1 = e_2, \ \ T^\ast e_2 = e_0 + e_3.$$
   That is
   $$ T e_1 = T^{\ast 2} e_1.$$
   One checks by recurrence that for all $k \geq 0,$ the vectors $e_{3k+1}$, $e_{3k+2}$, $e_{3k}+e_{3k+3}$ belong to the polynomial closure subspace $H_{\rm pol}$.
 We infer that the vectors $T^{\ast n} e_1$, $n \geq 0,$ generate the space $\ell^2(\N)$, which means that
   $T$ is an irreducible hyponormal operator and the exponential orthogonal polynomials are complete.
   
   Since $T$ is a Toeplitz operator with compact self-commutator, its essential spectrum coincides with the real algebraic curve
   $$ \Gamma = \{ z^2 + \frac{1}{z}; \  |z|=1\}$$
   and the spectrum of $T$ is the bounded set surrounded by $\Gamma$. As a matter of fact $\Gamma$ is a cycloid, in the form of three
   leaves, with a $2\pi/3$ rotational symmetry. Indeed, writing $z = e^{i \theta}$ for a point on the circle, the curve $\Gamma$
   is described by the parametric equations
   $$ x = \cos 2\theta + \cos \theta, \ \ y = \sin 2\theta - \sin \theta,$$
   or better for visualization:
   $$ x = 2 \cos \frac{3 \theta}{2} \cos \frac{\theta}{2}, \ \ y = 2 \cos \frac{3 \theta}{2} \sin \frac{\theta}{2},$$
   whence, in polar coordinates:
   $$ r = 2 |\cos \frac{3 \theta}{2}|,$$
   where $ \theta \in [0, 2 \pi]$.
   
   In conclusion the spectrum of the Toeplitz operator $T$ is a trifoil, with 3-fold rotational symmetry, consisting of three leaves intersecting at the origin.
   Their interiors are disjoint and carry a non-essential spectrum of Fredholm index equal to $-1$.
   
   The exponential orthogonal polynomials $P_n, \ n \geq 0,$ associated to $T$ satisfy therefore a four term relation,
   propagated from the first row
   $ T e_1 = e_3 + e_0 = T^{\ast 2} e_1$, that is:
   $$ 
   z P_k =    h_{k+1,k} P_{k+1} + h_{k,k} P_{k} + h_{k-1,k}P_{k-1} + h_{k-2,k} P_{k-2}, \ \ k \geq 2.
   $$
   Note that the associated Hessenberg matrix carries four non-zero diagonals, while the Toeplitz matrix representation of the same operator 
   has only two non-zero diagonals.

Note that one can equally consider the family of Toeplitz operators
$$ 
\alpha S^2 + \beta S^\ast, \ \ |\alpha| = |\beta|>0.
$$
If only aiming at producing a Toeplitz operator with rank-one self-commutator, one can look at
$$ 
V = \alpha S^{d+1} + \beta S^{\ast d},\ \ |\alpha| = |\beta|>0,
$$
where $d$ is a fixed positive integer. One proves as before that
$$ 
[V^\ast, V] = e_d \otimes e_d.
$$
Since 
$$ 
V e_d = e_{2d+1} + e_0,
$$ 
and
$$ 
(V^\ast)^N e_d = e_{(N+1)d} + {\rm lower \ order \ terms}, \ \ N \geq 0,
$$
we observe that in case $d>1$ there is no polynomial $G \in \C[z]$
with the property
$$ 
V e_d = G(V^\ast) e_d.
$$
Therefore, in this case, the Hessenberg matrix associated to the exponential orthogonal polynomials
is not finitely banded.

%%%%%%%%%%%%%%%%%%%%%%%%%%%%%%%%

\section{Rotationally invariant shapes}\label{Rot}

In this section we specialize our analysis to the case of a rotationally invariant principal function
$$ 
g(\zeta) = g_T(\zeta) = g_T(|\zeta|), \ \ \zeta \in \C.
$$
In terms of the associate hyponormal operator $T$, this invariance means that for every
$\theta \in [0,2\pi)$ there exists a unitary transformation $U_\theta : H \longrightarrow H$ satisfying
$$ 
e^{i\theta} T = U_\theta T U_\theta^\ast.
$$

Under this assumption the exponential transform simplifies:
$$ 
E_g(z,w) = \exp [\frac{-1}{\pi} \int \frac{g(r) r dr d\theta}{(re^{i\theta} - z)(re^{-i\theta} - \ol{w})}] = 
$$
$$ 
\exp [ \int_0^\infty \frac{ g(r) d(r^2)}{r^2 - z\ol{w}}] = \exp [ \int_0^\infty \frac{ g(\sqrt{t}) dt}{t - z\ol{w}}].
$$
These formulas, obtained via power expansion at infinity, are valid for large values of $|z|$ and $|w|$.
The additive representation of analytic functions with positive imaginary part in the upper half plane \cite{AD}
yields a positive measure $\nu$ supported by the semiaxis, with the property:
$$ 
\exp [ \int_0^\infty \frac{ g(\sqrt{t}) d(t)}{t - z\ol{w}}] = 1 + \int_0^\infty \frac{d \nu(s)}{s-z\ol{w}}.
$$
Since the original ``phase shift" $g(\sqrt{t})$ has compact support, the representing measure $\nu$
is also compactly supported. And vice-versa, any such measure corresponds to a $g(\sqrt{t})$, see again
\cite{AD}.

On the other hand, the pure hyponormal operator $T$ with principal function equal to $T$ satisfies:
$$
1-\langle (T^\ast - \ol{w})^{-1}\xi, (T^\ast - \ol{z})^{-1}\xi \rangle = 1 + \int_0^\infty \frac{d \nu(s)}{s-z\ol{w}},
$$
whence
$$ 
-\sum_{k=0}^\infty \frac{ \| T^{\ast k} \xi\|^2}{z^{k+1} \ol{w}^{k+1}} = \int_0^\infty \frac{d \nu(s)}{s-z\ol{w}}.
$$
From here we identify $\| T^{\ast k} \xi\|^2$, $k \geq 0,$ with the Stieltjes sequence of moments of a positive measure
with compact support on the semi-axis. And conversely, a compactly supported Stieltjes measure $\nu$ gives rise
through the above formulas to a unique rotationally invariant irreducible hyponormal operator $T$ with rank-one self-commutator.

For instance the unit disk carrying the Lebesgue measure ($g = \chi_{\D} dA$) corresponds to the Dirac mass $\nu=\delta_0$.
Notice that in this case $T^\ast \xi = 0$ and it is the only situation when the space
$$ 
H_{\rm pol} = \bigvee_{n \geq 0}  T^{\ast n}\xi
$$
is finite dimensional.

Assuming the space  $H_{\rm pol}$ is infinite dimensional, we turn to the
associated exponential orthogonal polynomials.  Due to the orthogonality conditions
$$ 
\langle T^{\ast k}\xi, T^{\ast m}\xi \rangle = 0, \ \ k\ne m,
$$ 
these are precisely
$$ 
P_n(z) = \frac{z^n}{\| T^{\ast n} \xi \|^{1/2}},\ \ n \geq 0.
$$
Thus the associated Hessenberg matrix ${\mathcal H}$ has only the first subdiagonal non-zero, say
$$ 
\mathcal{H} =  \left( \begin{array}{ccccc}
                0& 0 &0&0& \ldots\\
                \gamma_0&0&0&0& \ldots \\
                0&\gamma_1&0& 0& \ldots \\
                0&0& \gamma_2 & 0 & \ddots\\
                \vdots& & \vdots& & \ddots
                \end{array} \right).
$$
where $\gamma_k >0$ for all $k\geq 0$. 

If in addition to the subspace $H_{\rm pol}$ being infinite dimensional it coincides with the full Hilbert space $H$, then
$\mathcal{H}$ represents $T^\ast$  with respect to the basis $e_n=P_n(T^\ast)\xi$. But since $[T^\ast, T] = \xi \otimes \xi$ this gives
%If the closed subspace $H_{\rm pol}$ generated by $q(T^\ast)\xi$ with $q$ a polynomial coincides with the entire Hilbert space $H$ where $T$ acts,
%then ${\mathcal H}$ represents $T^\ast$ in the basis $e_n = P_n(T^\ast)\xi$, and consequently $[T^\ast, T] = \xi \otimes \xi$. But this implies
$$ 
- \gamma_0^2 = \frac{ \langle [T^\ast, T] \xi, \xi \rangle }{\|\xi\|^2} >0,
$$
a contradiction.

In short we have proved that for rotationally invariant hyponormal operators in our class, the space $H_{\rm pol}$ spanned by the exponential orthogonal polynomials 
never coincides with the full space $H$. We elaborate two simple examples.

\begin{example}\label{ex:tdisk}
Let $g=t\chi_\D$ for some $0<t<1$. Then
$$
E_g(z,w)=1-\langle (T^\ast -\bar{w})^{-1}\xi,(T^\ast -\bar{z})^{-1}\xi\rangle
$$ 
is to be identified with 
$$
(E_\D(z,w))^t=(1-\frac{1}{z\bar{w}})^t,
$$
for large $z$, $w$. This gives, after power expansions,
$$
\langle T^{\ast k} \xi,  T^{\ast k} \xi\rangle=(-1)^k \binom{t}{k+1}= \frac{t(1-t)\dots (k-t)}{(k+1)!},
$$
$$
\langle T^{\ast k} \xi,  T^{\ast j} \xi\rangle=0, \quad k\ne j.
$$

Denoting
$$
c_k= 1/\sqrt{(-1)^k \binom{t}{k+1}}>0
$$ 
we obtain the canonical ON-basis $e_0, e_1, e_2, \dots$ of the space $H_{\rm pol}$ given by
$$
\begin{cases}
e_0=c_0\,\xi,\\
e_1=c_1 T^\ast\xi,\\
e_2=c_2 T^{\ast 2}\xi,\\ 
e_3=c_3 T^{\ast 3}\xi,\\
\dots.
\end{cases}
$$
In terms of $\{e_0, e_1, e_2, \dots\}$ the compression to the polynomial subspace   of the cohyponormal operator $T^\ast$ is represented by
Hessenberg's matrix
$$ 
\mathcal{H} =  \left( \begin{array}{ccccc}
                0& 0 &0&0& \ldots\\
                c_0/c_1&0&0&0& \ldots \\
                0&c_1/c_2&0& 0& \ldots \\
                0&0& c_2/c_3 & 0 & \ddots\\
                \vdots& & \vdots& & \ddots
                \end{array} \right).
$$
However this compression is not hyponormal itself, only the full operator is.

To find the measure $\nu$ in the additive representation of the Nevanlinna analytic function
$$ 
f(z) = (1- \frac{1}{z})^t, \ \ z \notin [0,1],
$$
we have to compute, at the level of distributions
$$ 
\pi \nu(x) = \lim_{ \epsilon \rightarrow 0} \im f(x+i\epsilon),
$$
see \cite{AD}. In the limit $t=1$ this gives $\nu=\delta_0$, as expected, while for $0<t<1$
$\nu$ is absolutely continuous with density
$$
\frac{d\nu}{dx}=\frac{\sin (\pi t)}{\pi}\, (\frac{1}{x}-1)^t , \quad 0<x<1.
$$

\end{example}

\begin{example}\label{ex:annulus} 
The annulus also offers another computable example. Specifically,
for $g = \chi_{\D_R} - \chi_{\D_r}$, $0<r<R$, we find:
$$
E_g(z,w)=\frac{E_{\D_R}(z,w)}{E_{\D_r}(z,w)}= \frac{1-\frac{R^2}{z\bar{w}}}{1-\frac{r^2}{z\bar{w}}}
=(1-\frac{R^2}{z\bar{w}})(1+\frac{r^2}{z\bar{w}}+ \frac{r^4}{z^2 \bar{w}^2}+ \dots)
$$
$$
=1- (R^2-r^2) (\frac{1}{z\bar{w}}+\frac{r^2}{z^2\bar{w}^2}+ \frac{r^4}{z^3 \bar{w}^3}+ \dots),
$$
to be identified with
$$
1-\langle (T^\ast -\bar{w})^{-1}\xi,(T^\ast -\bar{z})^{-1}\xi\rangle=-\sum_{k,j\geq 0}\frac{\langle T^{\ast j}\xi, T^{\ast k}\xi \rangle}{z^{k+1}\bar{w}^{j+1}}.
$$
It follows that the canonical ON-basis is
$$
\begin{cases}
e_0=\frac{1}{\sqrt{R^2-r^2}}\,\xi,\\
e_1=\frac{1}{r\sqrt{R^2-r^2}}\,\xi,\\
e_2=\frac{1}{r^2\sqrt{R^2-r^2}}\,\xi,\\ 
e_3=\frac{1}{r^3\sqrt{R^2-r^2}}\,\xi,\\
\dots,
\end{cases}
$$
with the compression of $T^\ast$ to the polynomial subspace $H_{\rm pol}$
represented by the matrix
\begin{equation}\label{matrixH}
\mathcal{H} =  \left( \begin{array}{ccccc}
                0& 0 &0&0& \ldots\\
                1/r&0&0&0& \ldots \\
                0&1/r&0& 0& \ldots \\
                0&0& 1/r & 0 & \ddots\\
                \vdots& & \vdots& & \ddots
                \end{array} \right).
\end{equation}

In this case the additive representation of the corresponding Nevanlinna function
$$ 
f(z) = \frac{z-{R^2}}{z-{r^2}}
$$ 
is straightforward:
$$ 
f(z) = 1 + \int_\R \frac{d \nu(s)}{s-z},
$$ 
where $d\nu = (R^2-r^2)\delta_{r^2}.$

We also validate in this particular situation the general fact that Hessenberg matrix $\mathcal H$ together with the positive number
$ \| \xi \|^2 = \frac{R^2 - r^2}{\pi}$ determines the original shape. Obviously knowledge of the norm of $\xi$ is really needed here since
the matrix (\ref{matrixH}) gives no information about the radius $R$. 

Related to the latter statement is the fact that $H_{\rm pol}$ is strictly smaller than $H$, despite it being infinite dimensional. In fact, if
the two spaces had been equal, then $\mathcal{H}$ would have fully represented $T^\ast$ and all information of $\xi$ would have been built into it. 
In addition, $\mathcal{H}$ would have been cohyponormal, while the matrix in (\ref{matrixH}) clearly is (strictly) subnormal (and hyponormal).

One can see these features clearly in functional models of $H$, discussed in \cite{GP}, \cite{GP-resolvent}. Such models are generated by
the monomials $z^k\ol{z}^j$, $k,j\geq 0$, considered as functions defined on the spectrum of the operator. But these monomials are far from being 
linearly independent in general. For example, when $\dim H_{\rm pol}<\infty$ then there is a monic (say) polynomial of $P_d$ of degree $d=\dim H_{\rm pol}$
such $P_d(T^\ast)\xi=0$ (cf. Section~\ref{sec:completeness}), which means 
that the relation $P_d(z)=0$ holds in the relevant functional space. In this case ${\rm int\,}\sigma$ is a quadrature domain with quadrature nodes at the zeros of $P_d$. 

When $\sigma$ is the characteristic function of an ellipse it turns out that a relation $\ol{z}=a z$ holds for some constant $a\in\C$ (see Example~4.2 in \cite{GP-resolvent}),
and as for the present example, with the annulus, we have a relation
\begin{equation}\label{barzz}
\ol{z}=\frac{R^2+r^2}{2z}.
\end{equation}
As explained in \cite{GP-resolvent} (Theorem~4.1 there) such a relation comes out on exhibiting a continuous function $f$
in $\C$ which vanishes outside $\sigma$ and which in ${\rm int\,}\sigma$ satisfies
$$
\frac{\partial f}{\partial \ol{z}}= \ol{z}-\frac{R^2+r^2}{2z}.
$$
In the present case this function is
$$
f(z)=\frac{(|z|^2-r^2)(|z^2|-R^2)}{2z^2}.
$$

The relation (\ref{barzz}) means that instead of letting the functional space be generated by the two parameter family $z^k\ol{z}^j$ ($k,j\geq 0$),
it is enough to use the sequence $z^n$, with $n$ running through all integers (positive and negative). And these functions are then 
linearly independent. This is in full agreement with the general result (\cite{GP}, Theorem~3.1) saying 
that the analytic functions on $\sigma$ are densely injected in the functional model of $H$, in the non-quadrature domain case.

\end{example}

%%%%%%%%%%%%%%%%%%%%%%%%%%%%%%%%%%%%%%%%%%%%%%%%%%%%%%%%%%%%%%%

\section{Harmonic moments and Hele-Shaw flow}

The first column of the power moments, $\{a_{j0}, \, j\geq 0\}$, namely the ``{harmonic}'' moments, coincides with 
the first column of the exponential moments, $\{b_{j0},\, j\geq 0\}$, as is easily seen from (\ref{ab}). Hele-Shaw flow
moving boundary problems, or Laplacian growth (see \cite{Varchenko-Etingof-1992, Gustafsson-Teodorescu-Vasiliev-2014} for general theory), are 
characterized by these moments changing according to some simple law.
One case is the squeezing version of Hele-Shaw flow, meaning that a viscous fluid blob, represented by the principal
function $g$, is confined between two parallel plates  and one simply squeezes the plates together. The dynamics is then
that the harmonic moments decrease uniformly with respect to time $t$, say as
$$
a_{j0}(t)=b_{j0}(t)=C e^{-t}, \quad 0\leq t<\infty.
$$

Of major interest is also the inverse process, letting $t\to -\infty$ above whenever this is possible. This backward evolution is highly
unstable, but when it is successful it leads in the limit to potential theoretic skeletons, or ``mother bodies'', as discussed for
example in \cite{Savina-Sternin-Shatalov-2005, GP}. 

It is quite reasonable to consider principal functions $g=g(t)$  of fairly general form (subject only to $0\leq g\leq 1$). Indeed,
the equations describing Hele-Shaw flow are equivalent to those describing porous medium flow, namely
Darcy's law, and if $0<g<1$ in some region it simply means that the porous medium is not fully saturated with fluid there. 
Compare Example~\ref{ex:tdisk}.

However, the harmonic moments describe the dynamics accurately only when $g$ is of the form $\chi_\Omega$ 
with $\Omega$ a simply connected domain, otherwise they do not characterize $g$ completely. With $g$ of this form,
the most studied version of Hele-Shaw flow is that of injection or suction of fluid at one point, say the origin. In that case the dynamical
law is that $a_{00}$ changes linearly with time, while $a_{10}, a_{20},\dots$ remain fixed. This fits well into the recursion
algorithm in Corollary~\ref{cor:recursion} since the left member $b_{m+1,0}=a_{m+1,0}$ in (\ref{first-column}) then will be conserved
in time, as well as the factors $b_{0,n-j-1}=a_{0,n-j-1}$ in (\ref{all-columns}) (for $0\leq j\leq n-2$). 

It is however not clear at present how much of help this is for constructing solutions. Knowledge of all the $b_{kj}$ is equivalent to
knowledge of the complete exponential transform, and with the exponential transform at hand the boundary of the fluid domain is immediately obtained as
$\partial\Omega=\{z\in \C: E(z,z)=0\}$. Here $E(z,w)$ refers to the restriction of the exponential transform to the complement of
the spectrum $\ol\Omega$ and the analytic continuation of that across $\partial\Omega$ (known to exist in the present situation).
But the step from the first column of $(b_{kj})$ to the full matrix is a major challenge (and possible in principle only for  $g$ of the 
form $\chi_\Omega$ with $\Omega$ simply connected). 

One may also consider Hele-Shaw flow in exterior domains, containing the point of infinity. In this case one usually considers
suction at infinity, which means, in our setting, that the spectrum $\sigma$ of the hyponormal operator grows. This spectrum,
which hence is the complement of the fluid region in the Hele-Shaw model,
may represent, in other physical and mathematical models,  an electronic droplet in a Coulomb gas or a quantum Hall regime,
an aggregate of eigenvalues for a normal random matrix ensemble, a growing crystal in a phase transition model,
occupied sites for DLA (diffusion limited aggregation), etc. 
These subjects also connect to theories of integrable hierarchies, and here a superficial similarity with the 
material of the present paper is the occurrence of a pair of ``Lax'' operators, adjoint to each other and satisfying
a ``string equation'', actually a form of Heisenberg's uncertainty relation, which reminds of our commutation relation 
(\ref{TstarT}). In addition, the Lax operators are naturally represented  by Hessenberg matrices.   
We refer to  \cite{Wiegmann-Zabrodin-2000, Teodorescu-Bettelheim-Agam-Zabrodin-Wiegmann-2005, Mineev-Putinar-Teodurescu-2008} 
for some details on the above matters.

In the above inverse geometry, having the Hele-Shaw fluid in the complement of the spectrum, the usual harmonic moments $a_{j0}$
are no longer preserved. Instead it is the harmonic moments of the exterior domain that behave well.
In the notation of the mentioned papers related to integrable hierarchies these are
$$
t_k=\frac{1}{2\pi i k}\int_{\partial\sigma} z^{-k}\ol{z}dz \quad k=1,2,3,\dots,
$$
and they are interpreted as ``generalized times''. Ordinary time is then the moment $t=t_0=a_{00}$.

A good example is the ellipse, which is singled out in Theorem~\ref{ellipse} as the unique configuration admitting a three term relation for the 
exponential polynomials. An ellipse is characterized by $5$ real parameters, and these can be taken to be $t_0, t_1, t_2$ ($t_0$ is necessarily real).
The remaining exterior moments vanish: $t_3=t_4=\dots =0$ for any ellipse.

In the Hele-Shaw problem the ellipse shows up as a preserved shape in two ways.  One is with having the fluid in the exterior domain, with
injection or suction at infinity. This exterior problem means that the boundary moves with speed proportional to the density
of the equilibrium measure (for the interior domain), and in the ellipse case, taking the ellipses to be centered at the origin for simplicity, 
the resulting dynamics will be a family of  homothetic ellipses, defined via scale change in the complex plane. 
Under the evolution  $t_0$ equals the area of the ellipse divided by $\pi$, while $t_1$ and $t_2$ remain fixed (like $t_3=t_4=\dots=0$ of course).
This property of homothetic ellipses is related to the discovery of Newton that the shell between two such ellipses produces no gravitational field in the cavity inside the shell.

When the fluid is inside the ellipse it is the squeezing version of Hele-Shaw flow which preserves the elliptic shape. The family of
ellipses is however different, this time it is the family of confocal ellipses, i.e., the one-parameter family obtained by keeping the foci fixed.
This is one of few cases in which the dynamics of the squeezing problem works out well in the full range $-\infty<t<\infty$. In the limit $t\to-\infty$ one obtains 
the (unique in this case) mother body,  which is a measure supported by the segment between the foci.
In the standard case with the foci at $\pm 1$ the density of the measure with respect to arclength along the $x$-axis is proportional to $\sqrt{1-x^2}$.
For details on the above statements for ellipses we refer to \cite{Shapiro-1992, Varchenko-Etingof-1992, Khavinson-Lundberg-2018}.

It remains open how the three term relation for ellipses fits into the above pictures, but it has at least been proved that the zeros of the exponential polynomials
are attracted by the mother body: the counting measures of the zeros have a weak star limit (as the degrees of the polynomials tend to infinity), 
and this is a measure which has the same support as the mother body (however with a different density). See Proposition~7.1 in \cite{GP}.

%%%%%%%%%%%%%%%%%%%%%%%%%%%%%%%%%%


\begin{thebibliography}{88}
   
   \bibitem{AS} D. Aharonov, H.S. Shapiro, {\it Domains on which analytic functions satisfy quadrature identities}, J. Anal. Math {\bf 30}(1976), 39--73.
   
   \bibitem{AD} N. Aronsajn, W. F. Donoghue, {\it On exponential representations of analytic functions in the upper half-plane with positive imaginary part}, J. Analyse Math. {\bf 5}(1956), 321--388.
   
   \bibitem{B} B. Beckermann, {\it Complex Jacobi Matrices}, J. Comp. Appl. Math. {\bf 127}(2001), 17--65.
   
   \bibitem{BG} A. B\"ottcher, S. M. Grudsky, {\it Spectral Properties of Banded Toeplitz Matrices}, SIAM, Philadelphia, 2005.
   
   \bibitem{DXu} C. F. Dunkl, Yuan Xu, {\it Orthogonal Polynomials of Several Variables}, Second Edition, Encyclopedia of Mathematics and its Applications, vol. 155, Cambridge Univ. Press, Cambridge,
   2014.
   
   \bibitem{D} P. Duren, {\it Polynomials orthogonal over a curve}, Michigan Math. J. {\bf 12}(1965), 313--316.

  \bibitem{F} J. Favard, {\it Sur les polynomes de Tchebicheff}, C. R. Acad. Sci. Paris, {\bf 200}(1935), 2052--2053.

  \bibitem{GP-2003} B. Gustafsson, M. Putinar, {\it The exponential transform: a renormalized Riesz potential at crirical exponent}, 
    Indiana University Mathematical Journal {\bf 52}(2003), 527--568.
   
% \bibitem{GP4} B. Gustafsson, M. Putinar, {\it Linear analysis of quadrature domains. IV}, in vol. ``Quadrature Domains and their Applications" vol. {\bf 156}, Birkh\"auser, Basel, 2005, pp. 173-194.
   
   \bibitem{GP} B. Gustafsson, M. Putinar, {\it Hyponormal Quantization of Planar Domains}, Lect. Notes Math. {\bf 2199}, Springer, Cham, Switzerland, 2017.
   
\bibitem{GP-resolvent} B. Gustafsson, M. Putinar, {\it A field theoretic operator model and Cowen-Douglas class}, Banach J. Math., to appear.

\bibitem{Gustafsson-Teodorescu-Vasiliev-2014} {B. Gustafsson, R. Teoderscu, A. Vasil$'$ev},
  {\it Classical and Stochastic {L}aplacian growth}, {Advances in Mathematical Fluid Mechanics},
{Birkh\"auser Verlag}, {Basel}, {2014}.


\bibitem{Khavinson-Lundberg-2018} {D. Khavinson, E. Lundberg},
    {\it Linear Holomorphic Partial Differential Equations and Classical Potential Theory},
     {Mathematical Surveys and Monographs}, {232},
{American Mathematical Society, Providence, RI},  {2018}.
   
 \bibitem{KS} D. Khavinson, N. Stylianopoulos, {\it Recurrence relations for orthogonal polynomials and algebraicity of solutions of the Dirichlet problem}, 
    Around the Research of Vladimir Maz'ya II S, Partial Differential Equations (2009) Springer, Berlin, pp. 219--228.
   
\bibitem{MA} F. Marcell\'an, R. \'Alvarez-Nodarse, {\it On the "Favard Theorem" and its extensions}, Journal of Computational and Applied Mathematics {\bf 127}(2001),  231--254.

\bibitem{MP} M. Martin, M. Putinar, {\it Lectures on Hyponormal Operators}, Birkh\"auser Verlag, Basel, 1989.


\bibitem{Mineev-Putinar-Teodurescu-2008} {M. Mineev-Weinstein, M. Putinar, R. Teodorescu},
 {\it Random matrices in 2{D}, {L}aplacian growth and operator theory},
 {J. Phys. A},  {\bf 41}({2008}), {1--74}.
   
\bibitem{PS} M. Putinar, N. Stylianopoulos, {\it Finite-term relations for planar orthogonal polynomials}, Compl. Anal. Oper. Theory {\bf 1}(2007), 447--456. 

  % \bibitem{Sakai} M. Sakai, {\it Quadrature Domains}, Lect. Notes Math. {\bf 934}, Springer-Verlag, Berlin Heidelberg New York, 1982.

\bibitem{Savina-Sternin-Shatalov-2005}
  {T. Savina, B. Sternin, V. Shatalov},
     {\it On a minimal element for a family of bodies producing the same external gravitational field},
   {Appl. Anal.} {\bf 84}(2005), {649--668}.


\bibitem{Shapiro-1992} {H.S. Shapiro},
   {\it The {S}chwarz Function and its Generalization to Higher Dimensions},
   {University of Arkansas Lecture Notes in the Mathematical Sciences, 9}, {A Wiley-Interscience Publication},
{John Wiley \& Sons Inc.}, {New York}, {1992}.
     
\bibitem{Sz} G. Szeg\"o, {\it A problem concerning orthogonal polynomials}, Trans. Amer. Math. Soc. {\bf 37}(1935), 196--206.

\bibitem{Teodorescu-Bettelheim-Agam-Zabrodin-Wiegmann-2005}
{R. Teodorescu, E. Bettelheim, O. Agam, A. Zabrodin, P. Wiegmann},
{\it Normal random matrix ensemble as a growth problem},
 {Nuclear Phys. B}, {\bf 704}{(2005)}, {407--444}.

\bibitem{Varchenko-Etingof-1992} {A. Varchenko, P. Etingof},
{\it Why the Boundary of a Round Drop Becomes a Curve of Order Four},
{American Mathematical Society AMS University Lecture Series},
{Providence, Rhode Island}, {1992}.

\bibitem{Wiegmann-Zabrodin-2000} {P. Wiegmann, A. Zabrodin},
{\it Conformal maps and integrable hierarchies}, {Comm. Math. Phys.} {\bf 213}{(2000)}, {523--538}.


  
   
   \end{thebibliography}
\end{document}